\documentclass[10pt,draft,reqno]{amsart}
\usepackage{color}

     \makeatletter
     \def\section{\@startsection{section}{1}%
     \z@{.7\linespacing\@plus\linespacing}{.5\linespacing}%
     {\bfseries
     \centering
     }}
     \def\@secnumfont{\bfseries}
     \makeatother
\newtheorem{theorem}{Theorem}[section]
\newtheorem{lemma}[theorem]{Lemma}

\newtheorem{corollary}[theorem]{Corollary}
\theoremstyle{definition}
\newtheorem{definition}[theorem]{Definition}

\theoremstyle{remark}
\newtheorem{remark}[theorem]{Remark}
\numberwithin{equation}{section}
\begin{document}
\markboth{{P. Jin, B. R\"udiger and C. Trabelsi}}{Stochastics: An 
International Journal of Probability and Stochastic Processes}
\title[Positive Harris recurrence and exponential ergodicity  of the BAJD]{Positive Harris recurrence and exponential ergodicity  of the basic affine jump-diffusion}

\author[P. Jin]{Peng Jin}
\address{Peng Jin: Fachbereich C, Bergische Universit\"at Wuppertal, 42119 Wuppertal, Germany}
\email{jin@uni-wuppertal.de}

\author[B. R\"udiger]{Barbara R\"udiger}
\address{Barbara R\"udiger: Fachbereich C, Bergische Universit\"at Wuppertal, 42119 Wuppertal, Germany}
\email{ruediger@uni-wuppertal.de}

\author[C. Trabelsi]{Chiraz Trabelsi}
\address{Chiraz Trabelsi: Department of Mathematics, University of Tunis El-Manar, 1060 Tunis, Tunisia}
\email{trabelsichiraz@hotmail.fr}



\begin{abstract}
In this paper we find the transition densities of the basic affine jump-diffusion (BAJD), which is introduced by Duffie and G\^{a}rleanu [D. Duffie and N. G\^{a}rleanu, \emph{Risk and valuation of
  collateralized debt obligations}, Financial Analysts Journal 57(1) (2001), pp. 41--59] as an extension of the  CIR model with jumps. We  prove the positive Harris recurrence and exponential ergodicity of the BAJD. Furthermore we prove that the unique invariant probability measure $\pi$ of the BAJD is absolutely continuous with respect to the Lebesgue measure and we also derive a closed form formula for the density function of $\pi$.
 \end{abstract}

\maketitle

\section{Introduction}
In this paper we study the basic affine jump-diffusion (shorted as BAJD), which is given as  the unique strong solution $X:=(X_t)_{t \geq 0}$ to the following stochastic differential equation
\begin{equation}\label{jcir}
     dX_{t}=a(\theta-X_{t})dt +\sigma \sqrt{X_{t}}dW_{t}+dJ_t, \quad X_0 \geq 0,
 \end{equation}
where $a,\theta$, $\sigma$ are positive constants, $(W_{t})_{t \ge 0}$ is a $1$-dimensional Brownian motion and $(J_t)_{t \ge 0}$ is an independent $1$-dimensional pure-jump L\'evy process with the L\'evy measure
\[
\nu(dy)=\begin{cases}cde^{-dy}dy, &\ y \ge 0, \\ 0, & \ y<0,\end{cases}
\]
for some constants $c>0$ and $d>0.$ We assume that all the above processes are defined on some filtered probability space $(\Omega, \mathcal{F}, (\mathcal{F})_{t \ge 0}, P)$.

The process $X=(X_t)_{t \geq 0}$ given by (\ref{jcir}) has been introduced by Duffie and G\^{a}rleanu \cite{Duffie01} to describe the dynamics of default intensity.  It was also used in \cite{MR1850789} and \cite{MR2390186} as a short-rate model. Due to its simple structure, it is later referred as the basic affine jump-diffusion. The existence and uniqueness of strong solutions to the SDE (\ref{jcir}) follow from the main results of \cite{MR2584896}. We should remark that the BAJD process $X=(X_t)_{t \geq 0}$ in (\ref{jcir}) stays non-negative, given that $X_0 \ge 0$. This fact can be shown rigorously with the help of comparison theorems for SDEs,  for more details we refer the reader to \cite{MR2584896}. 

As its name implies, the BAJD belongs to the class of affine processes. Roughly speaking, affine processes are Markov processes for which the logarithm of the characteristic function of the process is affine with respect to the initial state.  Affine processes on the canonical state space $\mathbb{R}^m_+\times \mathbb{R}^{n}$ have been thoroughly investigated by Duffie et al \cite{MR1994043}, as well as in \cite{MR2851694}. In particular, it was shown in \cite{MR1994043} (see also \cite{MR2851694}) that any stochastic continuous affine process on $\mathbb{R}^m_+\times \mathbb{R}^{n}$ is a Feller process and a complete characterization of  its generator has been derived. Results on affine processes with the state space $\mathbb{R}_+$ can also be found in \cite{MR1850789}. 

Affine processes have found vast applications in mathematical finance, because of their complexity and computational tractability. As mentioned in \cite{MR1994043}, these applications include the affine term structure models of interest rates, affine stochastic volatility models, and many others.

Recently the long-term behavior of affine processes with the state space $\mathbb{R}_+$ has been studied in \cite{MR2390186} (see also \cite{MR2922631}), motivated by some financial applications in affine term structure models of interest rates. In particular they have found some sufficient conditions such that the affine process converges weakly to a limit distribution.  This limit distribution was later shown in \cite{MR2779872} as the unique invariant probability measure of the process.  Under further sharper assumptions it was even shown in \cite{2013arXiv1301.3243L} that the convergence of the law of the process to its invariance probability measure under the total variation norm is exponentially fast, which is called the exponential ergodicity in the literature. The method used in \cite{2013arXiv1301.3243L} to show the exponential ergodicity is based on some coupling techniques. 

In this paper we investigate the long-time behavior of the BAJD. More precisely, as the first main result of this paper, we show that the BAJD is positive Harris recurrent. As a well-known fact, Harris recurrence implies the existence of (up to the multiplication by a positive constant) unique invariant measure. Therefore our result on the positive Harris recurrence of the BAJD provides another way of proving the existence and uniqueness of invariant measures for the BAJD. Another consequence of the positive Harris recurrence is the limit theorem for additive functionals (see e.g. \cite[Theorem 20.21]{MR1876169}), namely $(1/t) \int_0^t f(X_s)ds$ converges almost surely to $\int_{\mathbb{R}_+}f(x)\pi(dx)$ as $t \to \infty$ for any $f \in \mathcal{B}_b(\mathbb{R}_+)$, where $\pi$ is the unique invariant probability measure for the BAJD. Some applications of Harris recurrence in statistics and calibrations of some financial models can be found in \cite{MR2002527} and \cite{MR3167406}.

As a key step in proving the positive Harris recurrence, we will derive a closed formula for the transition densities of the BAJD, which seems also to be a new result. In particular this formula indicates that the law of the BAJD process at any time $t>0$ is a convolution of a mixture of Gamma-distributions with a noncentral chi-square distribution. We should point out that this fact has already been discovered in \cite{MR1850789} for the BAJD in some special cases, namely for the BAJD whose parameters $(a,\theta, \sigma, c, d)$ satisfy $c/(a-\sigma^2d/2) \in \mathbb{Z}$ . In general, it is not easy to identify the transition density functions of affine processes. However, as shown in \cite{MR3084047}, a general density approximation procedure can be carried out for certain affine processes and in particular for the BAJD.  Different from \cite{MR3084047} we seek in this paper a closed formula of the transition densities of the BAJD. 

Finally we show the exponential ergodicity of the BAJD. We should indicate that the BAJD does not satisfy the assumptions required in \cite{2013arXiv1301.3243L} in order to get the exponential ergodicity. Our method is  also a different one and is based on the existence of a Foster-Lyapunov. 

The remainder of this paper is organized as follows. In Section 2 we collect some key facts on the BAJD. In Section 3 we introduce the so-called Bessel distributions and some mixtures of Bessel-distributions. In Section 4 we derive a closed formula for the transition densities of the BAJD. In Section 5 we first study some continuity properties of the transition densities of the BAJD and then show its positive Harris recurrence. In Section 6 we show the exponential ergodicity of the BAJD.

\section{Preliminaries}

In this section we recall some key facts on the BAJD. As defined in (\ref{jcir}), it is the unique strong solution $X=(X_t)_{t \geq 0}$ to the SDE
\[
     dX_{t}=a(\theta-X_{t})dt +\sigma \sqrt{X_{t}}dW_{t}+dJ_t, \quad X_0 \geq 0.
 \]

Throughout this paper we denote $E_x (\cdot)$ and $P_x (\cdot)$ as the expectation and probability respectively given the initial condition $X_0 = x $, with $x \ge 0$ being a constant.

By the affine structure of  the BAJD process $X$, the characteristic function of $X_t$ (given that $X_0=x$) is of the form
\begin{equation}\label{characfunction}
E_x\big[e^{uX_t}\big]=\exp \big (\phi(t,u)+x\psi(t,u)\big), \quad u \in \mathcal{U}:=\{u \in \mathbb{C}: \Re u \le 0\},
\end{equation}
where the functions $\phi(t,u)$ and $\psi(t,u)$ solve the generalized Riccati equations
\begin{equation}\label{eqaffine}
\begin{cases} \partial_t \phi(t,u)=F\big(\psi(t,u)\big), \quad & \phi(0,u)=0, \\
 \partial_t \psi(t,u)=R\big(\psi(t,u)\big), \quad & \psi(0,u)=u, \\
\end{cases}
\end{equation}
with
\begin{align*}
F(u)=a\theta u+\frac{cu}{d-u}, \quad & u \in \mathbb{C} \setminus \{d\}, \\
R(u)=\frac{\sigma^2u^2}{2}-au, \quad & u \in \mathbb{C}.
\end{align*}
By solving the system (\ref{eqaffine}) we get 
\begin{equation}\label{eqpsi}
\psi(t,u)=\frac{ue^{-at}}{1-\frac{\sigma^2}{2a}u(1-e^{-at})}
\end{equation}
and
\begin{equation}\label{eqphi}
\phi(t,u)=\begin{cases}-\frac{2a\theta}{\sigma^2}\log \big(1-\frac{\sigma^2}{2a}u(1-e^{-at})\big) \\ \qquad \quad  +\frac{c}{a-\frac{\sigma^2d}{2}}\log \Big (\frac{d-\frac{\sigma^2du}{2a}+\big(\frac{\sigma^2d}{2a}-1\big)ue^{-at}}{d-u}\Big ),  \hspace*{0.73cm}
\mbox{if} \ \Delta\neq 0,
 \\ -\frac{2a\theta}{\sigma^2} \log \big(1-\frac{\sigma^2}{2a}u(1-e^{-at})\big)+\frac{cu(1-e^{-at})}{a(d-u)}, \quad \ \ \  \mbox{if} \ \Delta=0,
\end{cases}
\end{equation}
where $\Delta=a - \sigma^2d/2$.

According to (\ref{characfunction}),  (\ref{eqpsi}) and (\ref{eqphi}),  the characteristic function of $X_t$ is given by
\begin{equation}\label{eqchara}
E_x[e^{uX_t}]=\begin{cases}\big(1-\frac{\sigma^2}{2a}u(1-e^{-at})\big)^{-\frac{2a\theta}{\sigma^2}}  \cdot \Big (\frac{d-\frac{\sigma^2du}{2a}+\big(\frac{\sigma^2d}{2a}-1\big)ue^{-at}}{d-u}\Big )^{\frac{c}{a-\frac{\sigma^2d}{2}}} \\ \qquad \ \ \cdot \exp \Big (\frac{xue^{-at}}{1-\frac{\sigma^2}{2a}u(1-e^{-at})} \Big),   \qquad  \mbox{if} \ \Delta\neq 0, \\ %
 \big(1-\frac{\sigma^2}{2a}u(1-e^{-at})\big)^{-\frac{2a\theta}{\sigma^2}} \cdot \exp \Big (\frac{cu(1-e^{-at})}{a(d-u)}+\frac{xue^{-at}}{1-\frac{\sigma^2}{2a}u(1-e^{-at})} \Big),\\  \hspace*{5.45cm} \mbox{if} \ \Delta = 0 .
\end{cases}
\end{equation}
Obviously $E_x\big[\exp(uX_t)\big]$ is continuous in $t \ge 0$ and thus the BAJD process $X$ is stochastically continuous.

\smallskip
We should point out that if we allow the parameter $c$ to be $0$, then the stochastic differential equation (\ref{jcir}) turns into 
\begin{equation}\label{cir}
     dZ_{t}=a(\theta-Z_{t})dt +\sigma \sqrt{Z_{t}}dW_{t}, \quad Z_0=x\geq 0.
 \end{equation}
To avoid confusions we have used $Z_t$ instead of $X_t$ here. The unique solution $Z:=(Z_t)_{t \ge 0}$ to (\ref{cir}) is the well-known Cox-Ingersoll-Ross (shorted as CIR) process and it holds 
\begin{equation}\label{eqcharaCIR}
E_x[e^{uZ_t}]
= \big(1-\frac{\sigma^2}{2a}u(1-e^{-at})\big)^{-\frac{2a\theta}{\sigma^2}} \cdot   \exp \Big (\frac{xue^{-at}}{1-\frac{\sigma^2}{2a}u(1-e^{-at})} \Big)  .
 \end{equation} 
In Section 4 we will find a distribution $\nu_t$ on $\mathbb{R}_+$ such that 
  \begin{equation}\label{eqcharaJ}
  \int_{\mathbb{R}+} e^{uy} \nu_t(dy) =\begin{cases} \Big (\frac{d-\frac{\sigma^2du}{2a}+\big(\frac{\sigma^2d}{2a}-1\big)ue^{-at}}{d-u}\Big)^{\frac{c}{a-\frac{\sigma^2d}{2}}}, \qquad   \mbox{if} \ \Delta \neq 0, \\
    \exp \Big (\frac{cu(1-e^{-at})}{a(d-u)} \Big), \hspace*{2.81cm}   \mbox{if} \ \Delta=0 .
    \end{cases}
    \end{equation}
Then it follows from (\ref{eqchara}),  (\ref{eqcharaCIR}) and (\ref{eqcharaJ}) that the distribution of the BAJD is the convolution of the distribution of the CIR process and  $\nu_t$. In light of this observation we can thus identify the transition probabilities $p(t,x,y)$ of the BAJD with  
\begin{equation}\label{eqconv}
p(t,x,y)=\int_{\mathbb{R}_+}f(t,x,y-z)\nu_t(dz),  \qquad x,y\ge0, \ t>0,
\end{equation}
where $f(t,x,y)$ denotes the transition densities of the CIR process.

\begin{remark} For a different way of representing the distribution of $X_t$ as a convolution we refer the reader to \cite{MR1850789}. In fact it was indicated in \cite[Remark 4.8]{MR1850789} that the distribution of any affine process on $\mathbb{R}_+$ can be represented as the convolution of two distributions on $\mathbb{R}_+$.
\end{remark}


\section{Mixtures of Bessel distributions}
To  find a distribution $\nu_t$ with the characteristic function of the form (\ref{eqcharaJ}) and study the distributional properties of  the BAJD, it is inevitable to encounter the Bessel distributions and mixtures of Bessel distributions. 

We start with a slight variant of the Bessel distribution defined in \cite[p.15]{MR2448332}. Suppose that $\alpha$ and $\beta$ are positive constants. A probability measure $\mu_{\alpha,\beta}$ on $\big (\mathbb{R}_+,\mathcal{B}(\mathbb{R}_+)\big)$ is called a Bessel distribution  with parameters $\alpha$ and $\beta$ if
\begin{equation}\label{besseldistru}
\mu_{\alpha,\beta}(dx)=e^{-\alpha}\delta_{0}(dx)+\beta e^{-\alpha-\beta x}\sqrt{\frac{\alpha}{\beta x}} \cdot I_1 (2\sqrt{\alpha \beta x})dx,
\end{equation}
where $\delta_0$ is the Dirac measure at the origin and $I_1$ is the modified Bessel function of the first kind, namely
\[
 I_1(r)=\frac{r}{2}\sum_{k=0}^{\infty}\frac{\big(\frac{1}{4}r^2\big)^k}{k ! (k+1)!}, \qquad r \in \mathbb{R}.
\]

Now we consider mixtures of Bessel distributions. Let $\gamma >0$ be a constant and define a probability measure $m_{\alpha,\beta,\gamma}$ on $\mathbb{R}_+$ as follows:
\[
m_{\alpha,\beta,\gamma}(dx):=\int_0^{\infty}\mu_{\alpha t, \beta}(dx)\frac{t^{\gamma-1}}{\Gamma(\gamma)}e^{-t}dt.
\]

Similar to \cite{MR2448332} we can easily calculate the characteristic function of $\mu_{\alpha,\beta}$ and $m_{\alpha,\beta,\gamma}$. 
\begin{lemma} 
For $u\in \mathcal{U}$ we have:
\begin{flalign*}
&(\mbox{i})  \quad \int_0^{\infty}e^{ux}\mu_{\alpha, \beta}(dx)= e^{\frac{\alpha u}{\beta-u}}. &\\
&(\mbox{ii}) \quad \int_0^{\infty}e^{ux}m_{\alpha, \beta,\gamma}(dx) =\Big (\frac{1}{\alpha+1}+\frac{\alpha}{\alpha +1}  \cdot \frac{1}{1-\frac{\alpha+1}{\beta}\cdot u} \Big)^{\gamma}.&
\end{flalign*}
\end{lemma}
\begin{proof}
(i) If $u\in \mathcal{U}$, then 
\begin{align*}
\int_0^{\infty}e^{ux}\mu_{\alpha, \beta}(dx) =&e^{-\alpha}+e^{-\alpha}\int_0^{\infty}\beta e^{-\beta x}\cdot e^{u x} \sqrt{\frac{\alpha}{\beta x}} \big(\sqrt{\alpha \beta x}\big)\cdot \sum_{k=0}^{\infty}\frac{(\alpha \beta x)^k}{k!(k+1)!}dx \\
=&e^{-\alpha}+e^{-\alpha}\int_0^{\infty}\alpha \beta e^{(u-\beta) x}\cdot  \sum_{k=0}^{\infty}\frac{(\alpha \beta x)^k}{k!(k+1)!} dx \\
=&e^{-\alpha}+\alpha \beta e^{-\alpha} \sum_{k=0}^{\infty} \int_0^{\infty} e^{(u-\beta) x} \frac{(\alpha \beta x)^k}{k!(k+1)!} dx \\
=&e^{-\alpha}+e^{-\alpha} \sum_{k=0}^{\infty} \Big (\frac{\alpha \beta}{\beta-u} \Big )^{k+1} \cdot \frac{1}{(k+1)!}  \\
=&e^{-\alpha} \sum_{k=0}^{\infty} \Big (\frac{\alpha \beta}{\beta-u} \Big )^{k} \cdot \frac{1}{k!}= e^{  \frac{\alpha u}{\beta-u}}.
\end{align*}
(ii) For $u\in \mathcal{U}$ we get
\begin{eqnarray}\label{eqfbessel}
 \int_0^{\infty}e^{ux}m_{\alpha, \beta,\gamma}(dx)  &=&  \int_{0}^{\infty}e^{\alpha t \cdot \frac{u}{\beta-u}}\cdot \frac{t^{\gamma-1}}{\Gamma(\gamma)}e^{-t}dt = \Big (1+\alpha \cdot \frac{u}{u-\beta} \Big)^{-\gamma} \nonumber \\
  &=& \Big ( \frac{-\beta +(\alpha+1)u}{-\beta+u} \Big)^{-\gamma}=\Big ( \frac{-\beta+u}{-\beta +(\alpha+1)u} \Big)^{\gamma} \nonumber \\
  &=&\Big ( \frac{\frac{1}{\alpha+1}((\alpha+1)u-\beta)+\frac{\beta}{\alpha+1}-\beta}{-\beta+(\alpha+1)u} \Big)^{-\gamma} \nonumber \\
   &=&  \Big (\frac{1}{\alpha+1}+\frac{\alpha}{\alpha +1}  \cdot \frac{1}{1-\frac{\alpha+1}{\beta}\cdot u} \Big)^{\gamma}.
\end{eqnarray}
\end{proof}

\begin{lemma} \label{lmixgamma} (i) The measure $m_{\alpha, \beta,\gamma}$ can be represented as follows:
\begin{equation}\label{mixgamma}
m_{\alpha,\beta,\gamma}(dx)=\Big (\frac{1}{1+\alpha} \Big)^{\gamma}  \delta _{0}+ g_{\alpha,\beta,\gamma}(x)dx, \quad x \ge0,
\end{equation}
 where
 \begin{equation}\label{defineg}
 g_{\alpha,\beta,\gamma}(x):= \sum_{k=1}^{\infty} \frac{\alpha^{k}\Gamma(k+\gamma)}{(\alpha+1)^{k+\gamma}\Gamma (\gamma)k!}\Gamma(x;k,\beta), \quad x \ge 0, 
\end{equation}
and $ \Gamma(x;k,\beta)$ denotes the density function of the Gamma distribution with parameters $k$ and $\beta$. \\
(ii) The function $g_{\alpha,\beta,\gamma}(x)$ defined in (\ref{defineg}) is a continuous function with variables $(\alpha,\beta,\gamma,x)\in D:=(0,\infty)\times(0,\infty)\times(0,\infty)\times[0,\infty)$.
\end{lemma}
\begin{proof}
(i) We can write
\begin{eqnarray*}
  m_{\alpha,\beta,\gamma}(dx) &=& \int_0^{\infty}\mu_{\alpha t, \beta}(dx)\frac{t^{\gamma-1}}{\Gamma(\gamma)}e^{-t}dt \\
  &=& \int_0^{\infty}\Big(e^{-\alpha t}\delta_{0}(dx)+\beta e^{-\alpha t-\beta x}\sqrt{\frac{\alpha t}{\beta x}} \cdot I_1 (2\sqrt{\alpha t\beta x})dx\Big)\frac{t^{\gamma-1}}{\Gamma(\gamma)}e^{-t}dt \\
  &=& \Big (\frac{1}{1+\alpha} \Big)^{\gamma}  \delta _{0}(dx)+\int_0^{\infty} \alpha \beta e^{-\alpha t-\beta x} \sum_{k=0}^{\infty} \frac{(\alpha t\beta x)^k}{k!(k+1)!}\cdot\frac{t^{\gamma}}{\Gamma(\gamma)}e^{-t}dtdx\\
  &=& \Big (\frac{1}{1+\alpha} \Big)^{\gamma}  \delta _{0}(dx)+ \alpha\beta e^{-\beta x}.\sum_{k=0}^{\infty}\Big(\int_0^{\infty}(\alpha\beta x)^k\frac{e^{-(\alpha+1)t}t^{\gamma+k}}{\Gamma(\gamma)k!(k+1)!}dt\Big)dx  \\
  &=&  \Big (\frac{1}{1+\alpha} \Big)^{\gamma}  \delta _{0}(dx)+  \sum_{k=0}^{\infty} \frac{\alpha^{k+1}\Gamma(k+\gamma+1)}{(\alpha+1)^{k+\gamma+1}\Gamma (\gamma)(k+1)!}\Gamma(x;k+1,\beta)dx \\
  &=&  \Big (\frac{1}{1+\alpha} \Big)^{\gamma}  \delta _{0}(dx)+ \sum_{k=1}^{\infty} \frac{\alpha^{k}\Gamma(k+\gamma)}{(\alpha+1)^{k+\gamma}\Gamma (\gamma)k!}\Gamma(x;k,\beta)dx.
\end{eqnarray*}

(ii) By the definition of $g_{\alpha,\beta,\gamma}(x)$ we have
\begin{eqnarray*}
  g_{\alpha,\beta,\gamma}(x) &=& \int_{0}^{\infty}\beta e^{-\alpha t-\beta x}\sqrt{\frac{\alpha t}{\beta x}} \cdot I_1 (2\sqrt{\alpha t \beta x})\frac{t^{\gamma-1}}{\Gamma(\gamma)}e^{-t}dt\\
   &=& \int_{0}^{\infty}\alpha \beta e^{-\alpha t-\beta x}   \big( \sum_{k=0}^{\infty}\frac{(\alpha t\beta x)^k}{k!(k+1)!} \big)\frac{t^{\gamma}}{\Gamma(\gamma)}e^{-t}dt\\
   &=& \int_{0}^{\infty} \frac{\alpha\beta t^\gamma}{\Gamma(\gamma)} e^{-(\alpha+1)t-\beta x}\big( \sum_{k=0}^{\infty}\frac{(\alpha t\beta x)^k}{k!(k+1)!} \big) dt
\end{eqnarray*}
Suppose that $(\alpha_0,\beta_0,\gamma_0,x_0)\in D$ and $\delta > 0$ is small enough such that $\gamma_0-\delta > 0$, $\alpha_0-\delta> 0$ and $\beta_0-\delta>0$. Then for $(\alpha,\beta,\gamma,x)\in K_\delta$ with \[
K_\delta:=\{(\alpha,\beta,\gamma,x)\in D: \max\{ |\alpha-\alpha_0|, |\beta-\beta_0|, |\gamma-\gamma_0|, |x-x_0|\}\le \delta\}\]
we get
\begin{align}\label{continuity}
 \frac{\alpha\beta t^\gamma}{\Gamma(\gamma)} e^{-(\alpha+1)t-\beta x}\big( \sum_{k=0}^{\infty}\frac{(\alpha t\beta x)^k}{k!(k+1)!} \big)  
  \leq& \frac{\alpha\beta t^\gamma}{\Gamma(\gamma)} e^{-(\alpha+1)t-\beta x}\big( \sum_{k=0}^{\infty}\frac{(\alpha t)^k(\beta x)^k}{(k!)^2} \big) \nonumber \\
   \leq & \frac{\alpha\beta t^\gamma}{\Gamma(\gamma)} e^{-(\alpha+1)t-\beta x}. e^{\alpha t}e^{\beta x}  \leq \frac{\alpha\beta t^\gamma}{\Gamma(\gamma)} e^{-t} \nonumber \\
   \leq &c_\delta\big(t^{\gamma_0-\delta}1_{[0,1]}(t)+t^{\gamma_0+\delta}e^{-t}1_{(1,\infty)}(t)\big)
\end{align}
for some constant $c_\delta >0$, since $\frac{\alpha\beta}{\Gamma(\gamma)}$ is continuous and thus bounded for \\ $(\alpha,\beta,\gamma,x)\in K_\delta$.
If $(\alpha_n,\beta_n,\gamma_n,x_n)\rightarrow (\alpha_0,\beta_0,\gamma_0,x_0)$  as $n\rightarrow\infty$, then by dominated convergence we get
\begin{equation*}
    \lim_{n\rightarrow\infty} g_{\alpha_n,\beta_n,\gamma_n}(x_n)=g_{\alpha_0,\beta_0,\gamma_0}(x_0),
\end{equation*}
namely $g_{\alpha,\beta,\gamma}(x)$ is a continuous function on $D.$
\end{proof}

\begin{remark}
If we write $\delta_0=\Gamma(0,\beta)$, namely considering the Dirac measure $\delta_0$ as a degenerated Gamma distribution, then the representation in (\ref{mixgamma}) shows that the measure $m_{\alpha,\beta,\gamma}$ is a mixture of Gamma distributions $\Gamma(k, \beta), \ k \in \mathbb{Z}_+$, namely 
\[
  m_{\alpha,\beta,\gamma}=\Big (\frac{1}{1+\alpha} \Big)^{\gamma}  \Gamma(0, \beta)+\sum_{k=1}^{\infty}  \frac{\alpha^{k}\Gamma(k+\gamma)}{(\alpha+1)^{k+\gamma}\Gamma (\gamma)k!} \Gamma(k, \beta).
 \]
\end{remark}
\section{Transition density of the BAJD}
In this section we shall derive a closed form expression for the transition density of the BAJD. We should mention that in \cite[Chapter 7]{MR1850789} the density functions of  the pricing semigroup associated to the BAJD was derived for some special cases. Essentially, the method used in \cite{MR1850789} could be used to derive the density functions of the BAJD in the case where $c/(a-\sigma^2d/2) \in \mathbb{Z}$. Here we proceed like \cite{MR1850789} but deal with more general parameters. In order to do this,  we first find, by using the results of the previous section, a probability measure $\nu_t$ on $\mathbb{R}_+$ whose characteristic function satisfies (\ref{eqcharaJ}).

 We recall that the BAJD process $X=(X_t)_{t \ge 0}$ is given by (\ref{jcir}). We distinguish between three cases according to the sign of $\Delta:=a-\sigma^2d/2$.

\subsection{Case \ i): $\Delta>0$}
From (\ref{eqchara}) we know that
\begin{align}
E_x[e^{uX_t}]=\big(1-\frac{\sigma^2}{2a}u(1-e^{-at})&\big)^{-\frac{2a\theta}{\sigma^2}} \cdot   \exp \Big (\frac{xue^{-at}}{1-\frac{\sigma^2}{2a}u(1-e^{-at})} \Big)  \nonumber \\
\cdot & \Big (\frac{d-\frac{\sigma^2du}{2a}+\big(\frac{\sigma^2d}{2a}-1\big)ue^{-at}}{d-u}\Big )^{\frac{c}{a-\frac{\sigma^2d}{2}}}   \label{ichara}
\end{align}
The product of the first two terms on the right-hand side of (\ref{ichara}) coincides with the characteristic function of the CIR process $Z=(Z_t)_{t \ge 0}$ defined in (\ref{cir}).  It is well-known that the transition density function of the CIR process is given by
\begin{equation}\label{cirdensity}
f(t,x,y) =\rho e^{-u-v}\Big ( \frac{v}{u} \Big)^{\frac{q}{2}} I_{q}\big (2(uv)^{\frac{1}{2}} \big)
\end{equation}
for $t>0, x>0$ and $y\ge0$, where
\begin{align*}
\rho \equiv & \frac{2a}{\sigma^{2}\Big (1-e^{-at}\Big)}, &u \equiv \rho xe^{-at},\\
v \equiv & \rho y, &q \equiv  \frac{2a\theta}{\sigma^{2}}-1,
\end{align*}
and $I_{q}(\cdot)$ is the modified Bessel function of the first kind of order $q$, namely 
\[
 I_q(r)=\big(\frac{r}{2}\big)^{q}\sum_{k=0}^{\infty}\frac{\big(\frac{1}{4}r^2\big)^k}{k ! \Gamma(q+k+1)}, \qquad r >0.
\]
We should remark that for $x=0$ the formula of the density function $f(t,x,y)$ given in (\ref{cirdensity}) is not valid any more. In this case we have 
\begin{equation}\label{cirdensity0}
f(t,0,y)= \frac{\rho}{\Gamma (q+1)}v^qe^{-v}
\end{equation}
for $t>0$ and $y \ge 0$.

Thus
\[
\int_{\mathbb{R}_+}f(t,x,y)e^{uy}dy= \big(1-\frac{\sigma^2}{2a}u(1-e^{-at})\big)^{-\frac{2a\theta}{\sigma^2}} \cdot   \exp \Big (\frac{xue^{-at}}{1-\frac{\sigma^2}{2a}u(1-e^{-at})} \Big)  .
\]
Now we want to find a probability measure $\nu_t$ with 
\begin{align} \int_{\mathbb{R}_+}e^{uy}\nu_t(dy)= &\Big (\frac{d-\frac{\sigma^2du}{2a}+\big(\frac{\sigma^2d}{2a}-1\big)ue^{-at}}{d-u}\Big )^{\frac{c}{a-\frac{\sigma^2d}{2}}} \nonumber \\
= &\Big (\frac{d-uL_1(t)}{d-u}\Big )^{\frac{c}{a-\frac{\sigma^2d}{2}}} \nonumber \\
=& \Big( L_1(t)+\big(1-L_1(t)\big) \frac{1}{1-\frac{u}{d}} \Big)^{\frac{c}{a-\frac{\sigma^2d}{2}}},\label{laplacenu}
\end{align}
where $L_1(t):=\exp(-at)+\sigma^2d\big( 1-\exp(-at) \big) /(2a)$. If such a measure $\nu_t$ exists, then the law of $X_t$ can be written as the convolution of the law of $Z_t$ and  $\nu_t$.

Comparing the characteristic functions (\ref{eqfbessel}) and (\ref{laplacenu}), it is easy to see that  we can seek the measure $\nu_t$ as a mixture of Bessel distributions. More precisely,  we define 
\begin{equation} \label{formulam}
\nu_t:=m_{\alpha_1(t),\beta_1(t),\gamma_1}
\end{equation}
with 
\begin{equation}\label{deltapositiv}
\begin{cases} \alpha_1(t): =\frac{1}{L_1(t)}-1 \\ \beta_1(t):  =\frac{d}{L_1(t)} \\
\gamma_1:=\frac{c}{a-\frac{\sigma^2d}{2}}.
\end{cases}
\end{equation}
Then the characteristic function of $\nu_t$ coincides with (\ref{laplacenu}). 
Since the probability measure $m_{\alpha_1(t),\beta_1(t),\gamma_1}$ is of  the form (\ref{mixgamma}), it follows now from  (\ref{eqcharaCIR}), (\ref{ichara}) and (\ref{laplacenu}) that the law of $X_t$ is absolutely continuous with respect to the Lesbegue measure and its density function $p(t,x,y)$ is given by 
\begin{equation}\label{jcirdensity}
p(t,x,y)=\Big (\frac{1}{1+\alpha_1(t)} \Big)^{\gamma_1} f(t,x,y)+ \int_0^y f(t,x,y-z)  g_{\alpha_1(t),\beta_1(t),\gamma_1}(z)dz
\end{equation}
for $t>0, \ x \ge 0$ and $y \ge0$,
where the function $g$ is defined in (\ref{defineg}). 
\subsection{Case \ ii): $\Delta<0$}
Similar to the case $(i)$, it suffices to find a probability measure $\nu_{t}$ with
\begin{eqnarray}
  \int_{\mathbb{R}_{+}} e^{uy}\nu_{t}(dy) &=& \Big (\frac{d-\frac{\sigma^2du}{2a}+\big(\frac{\sigma^2d}{2a}-1\big)ue^{-at}}{d-u}\Big )^{\frac{c}{a-\frac{\sigma^2d}{2}}}\nonumber \\
   &=& \Big (\frac{d-u}{d-\frac{\sigma^2du}{2a}+\big(\frac{\sigma^2d}{2a}-1\big)ue^{-at}}\Big )^{\frac{-c}{a-\frac{\sigma^2d}{2}}} \nonumber \\
   &=&  \Big (\frac{d-u}{d-L_1(t)u}\Big )^{\frac{-c}{a-\frac{\sigma^2d}{2}}} \nonumber \\
   &=&   \bigg (\frac{1}{L_1(t)}+(1-\frac{1}{L_1(t)}) \cdot \frac{1}{1-\frac{L_1(t)u}{d}}\bigg )^{\frac{-c}{a-\frac{\sigma^2d}{2}}}.\label{laplacenu2}
\end{eqnarray}
Since $\Delta=a-\sigma^2d/2<0$, therefore $\sigma^2d/2a>1$ and 
\[
L_1(t)=e^{-at}+\frac{\sigma^2d}{2a}\cdot \big( 1-e^{-at} \big)>1.\]
According to the formula ($\ref{besseldistru}$), we can choose
\begin{center} \label{formulam2}
    $\nu_t=m_{\alpha_2(t),\beta_2(t),\gamma_2}$
\end{center}
with the parameters $\alpha_2,\beta_2$ and $\gamma_2$ defined by
\begin{equation}\label{parameternegativ}
\begin{cases} \alpha_2(t):= L_1(t)-1 \\ \beta_2:= d \\
\gamma_2:=\frac{-c}{a-\frac{\sigma^2d}{2}}.
\end{cases}
\end{equation}
Similar to the case $(i)$, the transition densities $p(t,x,y)$ of $X$ is given by
\begin{equation}\label{densitydeltaneg}
p(t,x,y)=\Big (\frac{1}{1+\alpha_2(t)} \Big)^{\gamma_2} f(t,x,y)+ \int_0^y f(t,x,y-z)  g_{\alpha_2(t),\beta_2,\gamma_2}(z)dz
\end{equation}
for $t>0, \ x \ge 0$ and $y \ge0$, where the function $g$ is defined in (\ref{defineg}).

\subsection{Case \ iii): $\Delta=0$}
In this case we need to find a probability measure $\nu_t$ with
\begin{equation*}
    \int_{\mathbb{R}_{+}}e^{uy}\nu_{t}(dy)=\exp\Big(\frac{cu(1-e^{-at})}{a(d-u)}\Big).
\end{equation*}
According to the formula (\ref{mixgamma}) we can take $\nu_t$ as a Bessel distribution $\mu_{\alpha_3(t),\beta_3}$ with the parameters $\alpha_3(t)$ and $\beta_3$ defined by
\begin{equation}\label{paramDelta0}
  \begin{cases} \alpha_3(t) :=\frac{c}{a}(1-e^{-at}) \\ \beta_3:=d.
\end{cases}
\end{equation}
Thus in this case the transition densities $p(t,x,y)$ of $X$ is given by
\begin{align}
    p(t,x,y)=\int_{0}^{y}f(t,x,&y-z) \beta_3 e^{-\alpha_3(t)-\beta_3 z}\sqrt{\frac{\alpha_3(t)}{\beta_3 z}}I_1(2\sqrt{\alpha_3(t)\beta_3 z})dz \nonumber  \\
   +& e^{-\alpha_3(t)}f(t,x,y)\label{desitydelta0}
\end{align}
for $t>0, \ x \ge 0$ and $y \ge0$.

Summarizing the above three cases we get the following theorem. 

\begin{theorem}Let $X=(X_t)_{t \ge 0}$ be the BAJD defined in (\ref{jcir}). Then the law of $X_t$ given that $X_0=x \ge 0$ is absolutely continuous with respect to the Lesbegue measure and thus posseses a density function $p(t,x,y)$, namely 
\[
P_x(X_t \in A) = \int_{A}p(t,x,y)dy, \quad  t \ge 0, \ A \in \mathcal{B}(\mathbb{R}_+)
\]
According to the sign of $\Delta=a-\sigma^2d/2$, the density $p(t,x,y)$ is given by (\ref{jcirdensity}), (\ref{densitydeltaneg}) and (\ref{desitydelta0}) respectively. 
\end{theorem}

Although the density functions in (\ref{jcirdensity}), (\ref{densitydeltaneg}) and (\ref{desitydelta0}) are essentially different, they do share some similarities. In the following corollary we give a  unified representation of $p(t,x,y)$. 
\begin{corollary}Irrelevant of the the sign of $\Delta=a-\sigma^2d/2$, the transition densities $p(t,x,y)$ of $X$ can be expressed in a unified form as 
\begin{equation}\label{transition}
p(t,x,y)= L(t)f(t,x,y)+\int_{0}^{y}f(t,x,y-z)h(t,z)dz,
\end{equation}
where $L(t)$ is continuous function in $t >0$ which satisfies $0<L(t)<1$ for $t>0$, the function $h(t,z)$ is non-negative and  continuous in $(t,z)\in (0,\infty)\times[0,\infty)$ and satisfies $\int_{\mathbb{R}_+}h(t,z)dz=1-L(t).$
\end{corollary}


\section{Positive Harris recurrence of the BAJD}

It was shown in \cite{MR1994043} (see also \cite{MR2851694}) that the semigroup of any stochastically continuous affine process on the canonical state space $\mathbb{R}^m_+\times \mathbb{R}^{n}$ is a Feller semigroup. Define the semigroup of the BAJD by 
\begin{equation}
T_t f (x):= \int_{\mathbb{R}_{+}} p(t,x,y) f(y)dy,
\end{equation}
where $f: \mathbb{R}_{+} \to \mathbb{R}$ is bounded. Since the BAJD process $X$ is stochastically continuous and affine, thus $(T_t)_{t \ge 0}$ is a Feller semigroup. 

To show the positive Harris recurrence, we need first to prove the regularity property of BAJD. To this aim,  we first analyse the continuity properties of the integral which appears on the right hand side of (\ref{transition}).

\begin{lemma}\label{contF}
Let $f(t,x,y)$ be the transition density of the CIR process given in (\ref{cirdensity}) and $h(t,z)$ be the same as in ($\ref{transition}$). Then the function $F(t,x,y)$ defined by
\begin{equation}\label{defineF}
    F(t,x,y):=\int_{0}^{y}f(t,x,y-z)h(t,z)dz
\end{equation}
is continuous with variables $(t,x,y)\in(0,\infty)\times[0,\infty)\times[0,\infty)$. Moreover if $M>1$ is a constant, then  
 \begin{equation}\label{inqF}
 |F(t,x,y)|\leq C y^{\frac{2a\theta}{\sigma^{2}}} 
 \end{equation}
 for all 
\[
(t,x,y)\in K_{M}:=\{(t,x,y): \frac{1}{M} \leq t\leq M, \ 0\leq x\leq M,  \ 0\leq y \leq\frac{1}{M}\},\]
where $C >0$ is a constant which depends on $M$.\\

\end{lemma}

\begin{proof}
For simplicity we set $q:=2a\theta/\sigma^{2}-1$ as in (\ref{cirdensity}). Since $h(t,z)$ is continuous in $(t,z)\in(0,\infty)\times[0,\infty)$, thus there exists a constant $c_1>0$ depending on  $M$ such that
\begin{equation}\label{inqh}
    |h(t,z)|\leq c_1 \quad \mbox{for} \quad\frac{1}{M}\leq t\leq M, \quad 0\leq z \leq \frac{1}{M}.
\end{equation}
Therefore if $(t,x,y)\in K_{M}$, we have
\begin{equation}\label{AbsolutF}
    |F(t,x,y)|\leq c_1\int_{0}^{y}f(t,x,y-z)dz.
\end{equation}
According to (\ref{cirdensity}) and (\ref{cirdensity0}) we have
\begin{equation}\label{absolutdensity}
    |f(t,x,y-z)|\leq c_2|y-z|^q \quad \quad \mbox{if} \quad (t,x,y-z)\in K_{M} \ \mbox{ and  } \ y \neq z,
\end{equation}
where $c_2>0$ is a constant depending on $M$.\\
It follows from (\ref{AbsolutF}) and (\ref{absolutdensity}) that 
\begin{equation}\label{Festi}
    |F(t,x,y)|\leq c_1c_2\int_{0}^{y}|y-z|^qdz=\frac{c_1c_2}{q+1} y^{q+1} \le c_3y^{q+1} 
\end{equation}
for $(t,x,y)\in K_{M}$, if we set $c_3:=c_1c_2/(q+1) $. Thus (\ref{inqF}) is proved. Noting that $F(t,x,0)=0$ for $t>0$ and $x\geq 0$, the continuity of the function $F$ at points $(t,x,0)$ is an immediate consequence of the estimate (\ref{Festi}). \\
We now proceed to prove the continuity of $F$ at other points. Suppose that $t_0> 0$, $x_0\geq 0$ and $y_0> 0$ are fixed. Let $\epsilon >0$ be arbitrary. We choose $\delta_{1}>0$ small enough such that $y_0-2\delta_{1}>0$ and $t_0-\delta_{1}> 0$. As in (\ref{inqh}) and (\ref{absolutdensity}) there exist constants $c_4, c_5>0$, which depend on $\delta_1$, such that
\begin{equation}\label{inq2h}
    |h(t,z)| \leq c_4 \quad \mbox{for} \quad t\in[t_0-\delta_1,t_0+\delta_1] , \ \ z\in [0, y_0+\delta_1]
\end{equation}
and
\begin{equation}\label{inq2transionfun}
    |f(t,x,y-z)|\leq c_5 |y-z|^q
\end{equation}
$\ for \ t\in[t_0-\delta_1,t_0+\delta_1], \ x\in[0, x_0+\delta_1]$ and $0< y-z\leq y_0+\delta_1.$
Set
\begin{equation*}\label{Kdelta2}
  K_{\delta_2}:=[t_0-\delta_{2},t_0+\delta_{2}]\times[0,x_0+\delta_{2}]\times[y_0-\delta_{2},y_0+\delta_{2}].
\end{equation*}
We choose $\delta_2>0$ small enough such that $\delta_2<\delta_1$ and $c_4c_5(3\delta_2)^{q+1}/(q+1)< \epsilon/3$. If $(t,x,y)\in K_{\delta_2}$ then it holds
\begin{eqnarray}
 \Big |\int_{y_0-2\delta_2}^{y}f(t,x,y-z)h(t,z)dz\Big| &\leq & c_4c_5\int_{y_0-2\delta_2}^{y}(y-z)^q dz \nonumber \\
   &=& \frac{c_4c_5}{q+1}(y-y_0+2\delta_2)^{q+1} \nonumber \\
  &\leq & \frac{c_4c_5}{q+1}(3\delta_2)^{q+1} <\frac{\epsilon}{3}. \label{Inqintegral}
\end{eqnarray}
If $(t,x,y)\in K_{\delta_2}$ and $0\leq z \leq y_0-2\delta_{2}$, then $\delta_{2}\leq y-z\leq y_0+\delta_{2}$ and by (\ref{inq2h}) and (\ref{inq2transionfun}) we have
\begin{eqnarray}\label{inq2fh}
       |f(t,x,y-z)h(t,z)| &\leq & c_4c_5|y-z|^q \nonumber \\
        &\leq& c_4c_5 (|\delta_2|^q+|y_0+\delta_2|^q).
     \end{eqnarray}
Since for fixed $z\in[0,y_0-2\delta_{2}]$ the function $f(t,x,y-z)h(t,z)$ is continuous in $(t,x,y)\in K_{\delta_2}$, it follows from $(\ref{inq2fh})$ and dominated convergence theorem that
\begin{equation*}
    F_2(t,x,y):=\int^{y_0-2\delta_2}_{0}f(t,x,y-z)h(t,z)dz
\end{equation*}
is a continuous function in $(t,x,y)\in K_{\delta_2}$. This implies the existence of a constant $\delta$ with $0<\delta <\delta_2$ such that 
\begin{equation}\label{continuityF2}
    |F_2(t,x,y)-F_2(t_0,x_0,y_0)|< \frac{\epsilon}{3},
\end{equation}
if $(t,x,y)\in K_\delta:=[t_0-\delta,t_0+\delta]\times[0 \vee(x_0-\delta),x_0+\delta]\times[y_0-\delta,y_0+\delta]$. Thus it follows from $(\ref{Inqintegral})$ and $(\ref{continuityF2})$ that
\begin{equation*}
      |F(t,x,y)-F(t_0,x_0,y_0)|< \epsilon
\end{equation*}
for $(t,x,y)\in K_\delta$.
The continuity of the function $F$ at $(t_0,x_0,y_0)$ is proved.
\end{proof}


Following \cite[Chapter 20]{MR1876169} we give the definition of a regular Feller process on $\mathbb{R}_{+}$.
\begin{definition}
Consider a Feller process $Y=(Y_t)_{t \geq 0}$ with state space $\big(\mathbb{R}_{+}, \mathcal{B}(\mathbb{R}_{+}) \big)$ and distributions $P_x$, $x \in \mathbb{R}_+$. The process is said to be regular if there exist a locally finite measure $\rho$ on $\mathbb{R}_{+}$ and a continuous function $(t,x,y) \mapsto p_{t}(x,y)>0$ on $ (0,\infty)\times \mathbb{R}_{+}^{2}$ such that
\[
P_{x}\{Y_t \in B\}=\int_{B}p_t(x,y)\rho (dy), \quad x \in \mathbb{R}_{+}, \ B \in \mathcal{B}(\mathbb{R}_{+}), \ t>0.
\]
\end{definition}
\noindent The measure $\rho$ is called the "supporting measure" of the process. It is unique up to an equivalence (see \cite[page 399]{MR1876169}). \\

Proceeding as in the proof of \cite[ Proposition 2.2]{MR3167406}, we define a new measure $\eta$ on $\big(\mathbb{R}_{+}, \mathcal{B}(\mathbb{R}_{+}) \big)$ as
\begin{equation}\label{measureeta}
\eta(dx):=\kappa(x)dx,
\end{equation}
where
\[
\kappa(x)=\begin{cases} x^{\frac{2a \theta}{\sigma^{2}}-1}, \quad & 0 \le x \le 1, \\ 1, \quad & x >1.
\end{cases}
\]
 Then the transition densities of the BAJD process with respect to the new measure $\eta$ is given by
\begin{equation}\label{transition eta}
\tilde{p}(t,x,y) = \frac{p(t,x,y)}{\kappa(y)}, \quad t>0, \ x\ge 0,\  y >0.
\end{equation}

\begin{theorem}
The transition densities $\tilde{p}(t,x,y)$ of the BAJD process with respect to the measure $\eta$ satisfies
\[
0<\tilde{p}(t,x,y)<\infty, \quad t>0,\ x \ge 0,\ y\ge 0\]
and is continuous in $(t,x,y) \in (0,\infty)\times [0,\infty) \times [0,\infty)$. Thus the BAJD is a regular Feller process on the state space $\mathbb{R}_+$.
\end{theorem}

\begin{proof}
From (\ref{transition}) we can write the transition density of the BAJD process with respect to the measure $\eta$ as
\begin{eqnarray}
 \tilde{p}(t,x,y)  &=& L(t)\frac{f(t,x,y)}{\kappa(y)}+ \frac{F(t,x,y)}{\kappa(y)} \nonumber  \\
   &=& L(t) \tilde{f}(t,x,y) +\frac{F(t,x,y)}{\kappa(y)} \label{transition wrt eta}
\end{eqnarray}
where $\tilde{f}(t,x,y)$ is the transition density of CIR with respect the new measure $\eta$ and $F$ is defined in (\ref{defineF}). We have already shown in \cite[Proposition 2.2]{MR3167406} that \begin{equation}\label{positivef}
0<\tilde{f}(t,x,y) <\infty, \quad t>0, \ x\ge 0,\  y \ge 0,
\end{equation}
and $(t,x,y) \mapsto \tilde{f}(t,x,y)$ is a continuous function on $(0,\infty) \times \mathbb{R}_{+}^{2}$.\\
From Lemma \ref{contF} we know that the function $F(t,x,y)$ appearing in the second summand in (\ref{transition wrt eta}) is continuous on $(0,\infty) \times \mathbb{R}_{+}^{2}$ and
\begin{equation*}
|F(t,x,y)|\leq C y^{\frac{2a \theta}{\sigma^{2}}} , \qquad \mbox{if} \ (t,x,y)\in K_{M},
\end{equation*}
where $C>0$ is a constant depending on $M$.
Now it is clear that
\begin{equation*}
    0\leq \lim_{y\rightarrow 0} \frac{F(t,x,y)}{\kappa (y)} \leq \lim_{y\rightarrow 0} C|y|=0.
\end{equation*}
Since $F(t,x,0)=0$ and $L(t)$ is continuous in $t>0$, it follows that
the function $ \tilde{p}(t,x,y)$ is continuous at points $(t_0,x_0,0)$, if $t_0>0$ and $x_0 \ge 0$. The continuity of the function $ \tilde{p}(t,x,y)$ at other points is also clear, because all the functions appearing in (\ref{transition wrt eta}) are continuous and $0<\kappa(y)<\infty$ for $y>0$. Noting that $ 0<L(t)<1$ for $t> 0$, we get
\begin{equation*}
0<\tilde{p}(t,x,y) <\infty \quad \mbox{for all } \ (t,x,y) \in (0,\infty) \times \mathbb{R}_{+}^{2}.
\end{equation*}
Therefore the BAJD process is a regular Feller process with $\eta$ as the supporting measure.
\end{proof}

For the following definitions we refer to \cite{MR1234294} .
\begin{definition}\label{defiharris} Consider a time-homogeneous Markove process $Y=(Y_t)_{t \geq 0}$ with the state space $\mathbb{R}_{+}$ and distributions $P_x$, $x \in \mathbb{R}_+$. \\
(i) $Y$ is said to be Harris recurrent if for some $\sigma$-finite measure $\mu$
\begin{equation}\label{eqharrisdefi}
P_x \big(\int_{0}^{\infty} \mathbf{1}_{A}(Y_s)ds = \infty\big)=1,
\end{equation}
for any $x \in \mathbb{R}_{+}$ and $A \in \mathcal{B}(\mathbb{R}_{+})$ with $\mu(A)>0$. It was shown in \cite{MR580144} that if $Y$ is Harris recurrent then it possesses a unique (up to a renormalization) invariant measure. If the invariant measure is finite, then the process $Y$ is called positive Harris recurrent. \\
(ii) $Y$ is said to be uniformly transient if
\begin{equation}\label{transientine}
\sup_{x}E_{x}\Big [\int_{0}^{\infty} \mathbf{1}_{K}(Y_s)ds \Big ]<\infty
\end{equation}
for every compact $K \subset \mathbb{R}_{+}$.
\end{definition}
\begin{lemma}\label{nottrasient}
The BAJD is not uniformly transient.
\end{lemma}
\begin{proof} Let $m>0$, $K:=[0,m]$ and $x\in(0,\infty)$ be fixed. Then
\begin{eqnarray*}
  E_x \Big [ \int_{0}^{\infty}\mathbf{1}_{[0,m]}(X_t) dt \Big ] &=&  \int_{0}^{\infty}E_x\big[\mathbf{1}_{[0,m]}(X_t)\big] dt  \\
   &=& \int_{0}^{\infty}\int_{0}^{m} p(t,x,y) dy dt\\
  &=& \int_{0}^{m} dy \int_{0}^{\infty} p(t,x,y) dt.
  \end{eqnarray*}
From (\ref{transition}) and (\ref{defineF}) we know that 
\[
p(t,x,y)=L(t)f(t,x,y)+F(t,x,y).
\]
It follows from (\ref{jcirdensity}), (\ref{densitydeltaneg}) and (\ref{desitydelta0}) that 
\begin{equation}\label{limitL}
0<\lim_{t \to \infty} L(t) =\begin{cases}\big(\frac{\sigma^2d}{2a}\big)^{\frac{c}{a-\frac{\sigma^2d}{2}}} , \quad  \mbox{if} \ \Delta \neq 0, \\
 e^{-\frac{c}{a}}, \quad\quad\quad\quad \  \mbox{if} \ \Delta=0 .
\end{cases}
\end{equation}
Since $F(t,x,y)$ is non-negative, thus there exists large enough $T>0$ such that 
$$p(t,x,y) \ge \lambda  f(t,x,y) \quad \mbox{for } \ t \ge T,$$
where $\lambda>0$ is a constant. Let $\epsilon >0$ be small enough. According to \cite[Lemma 2.4]{MR3167406} we know that for any $x>0$ and $y\in [\epsilon, m]$ it holds 
$$\int_0^\infty f(t,x,y)dt=\infty.$$ Therefore 
\begin{align*}
E_x\Big [\int_{0}^{\infty}\mathbf{1}_{[0,m]}(X_t) dt\Big ]= &\int_{0}^{m} dy \int_{0}^{\infty} p(t,x,y) dt\\
\ge & \int_{\epsilon}^{m} dy \int_{R}^{\infty} \lambda f(t,x,y)dt=\infty.
\end{align*}
This proves that the BAJD is not uniformly transient.
\end{proof}
\begin{theorem} \label{ThHarris}
The BAJD is Harris recurrent.
\end{theorem}
\begin{proof}
We have shown that the BAJD is a regular Feller process with the measure $\eta$, which is defined in (\ref{measureeta}), as a supporting measure. It follows from Lemma \ref{nottrasient} and \cite[Theorem 20.17]{MR1876169} that the BAJD is Harris recurrent and the measure $\eta$ satisfies (\ref{eqharrisdefi}).
\end{proof}

\begin{remark}\label{reminv}
Since the BAJD is a Harris recurrent Feller process, it follows from \cite[Theorem 20.18]{MR1876169} that the BAJD has a locally finite invariant measure $\pi$ which is equivalent to the supporting measure $\eta$. Moreover every $\sigma$ -finite invariant measure of the BAJD agrees with $\pi$ up to a renormalization. The existence and uniqueness of an invariant probability measure for the BAJD has already been proved in \cite{MR2779872} (see also \cite{MR2390186}). Thus we can assume $\pi$ to be a probability measure. The characteristic function of $\pi$ was given in \cite{MR2390186} and has the form 
\begin{equation}\label{eqcharainv}
\int_{\mathbb{R}_+}e^{uz} \pi(dz)=\begin{cases}\big(1-\frac{\sigma^2}{2a}u\big)^{-\frac{2a\theta}{\sigma^2}}  \cdot \Big (\frac{d-\frac{\sigma^2du}{2a}}{d-u}\Big )^{\frac{c}{a-\frac{\sigma^2d}{2}}}, \qquad  \  \mbox{if} \ \Delta \neq 0, \\
 \big(1-\frac{\sigma^2}{2a}u\big)^{-\frac{2a\theta}{\sigma^2}} \cdot \exp \Big (\frac{cu}{a(d-u)}\Big),\ \ \qquad  \ \  \mbox{if} \ \Delta = 0. 
\end{cases}\\
\end{equation}
\end{remark}

\begin{corollary}\label{coroinv}
The BAJD is positive Harris recurrent. Its unique invariant probability measure $\pi$ is absolute continuous with respect to the Lebesgue measure and thus has a density function $l(\cdot)$, namely $\pi (dy)=l(y)d y, \ y \in \mathbb{R}_+$. If $\Delta \neq 0$, then we have
\[
l(y)=\Big(\frac{\sigma^2d}{2a}\big)^{\frac{c}{a-\frac{\sigma^2d}{2}}}\Gamma \big(y;\frac{2a\theta}{\sigma^2}, \frac{\sigma^2}{2a}\big)+\int_{0}^{y}\Gamma \big(y-z; \frac{2a\theta}{\sigma^2}, \frac{\sigma^2}{2a}\big)h(z)dz, \quad y \ge 0,
\]
where $ \Gamma \big(y; 2a\theta/\sigma^2, \sigma^2 / (2a)\big)$ denotes the density function of the Gamma distribution with parameters $2a\theta/\sigma^2$ and $\sigma^2 / (2a)$, and
\[
h(z)=\begin{cases}g_{\frac{2a}{\sigma^2d}-1,\frac{2a}{\sigma^2},\frac{c}{a-\frac{\sigma^2d}{2}}}(z),  \quad \ \qquad  \rm{if} \ \Delta > 0, \\
g_{\frac{\sigma^2d}{2a}-1,d,-\frac{c}{a-\frac{\sigma^2d}{2}}}(z), \ \ \qquad  \ \  \rm{if} \ \Delta < 0,
\end{cases}\\
\]
with $g$ defined in (\ref{defineg}). 
 If $\Delta = 0$, then we have 
 \[
 l(y)=e^{-\frac{c}{a}}\Gamma \big(y; \frac{2a\theta}{\sigma^2}, \frac{\sigma^2}{2a}\big)+\int_{0}^{y}\Gamma \big(y-z; \frac{2a\theta}{\sigma^2}, \frac{\sigma^2}{2a}\big) d e^{-\frac{c}{a}-dz}\sqrt{\frac{c}{a d z}} \cdot I_1 \big(2\sqrt{\frac{cdz}{a}}\big)dz
\]
for $y \ge 0$.
\end{corollary}
\begin{proof}
We apply the same method which we used in Section 4 to find the transition densities of the BAJD. Since the characteristic function of $\pi$ is given by (\ref{eqcharainv}) and noting that the first term on the right hand side of (\ref{eqcharainv}) corresponds to the characteristic function of a Gamma distribution, we can represent the measure $\pi$ as a convolution of a Gamma distribution with a probability measure $\nu$. If $\Delta=0$ the measure $\nu$ is a Bessel distribution, otherwise it is a mixture of Bessel distributions. By identifying the parameters of the Gamma distribution and the Bessel or mixture of Bessel distributions, we get an explicit formula of the density function of $\pi$. 
\end{proof}

\begin{corollary}
Let $X=(X_t)_{t \ge 0}$ be the BAJD defined by (\ref{jcir}).  Then for any $f \in \mathcal{B}_b(\mathbb{R}_+)$ we have 
 \[
 \frac{1}{t}\int_0^t f(X_s)ds \ \to \ \int_{\mathbb{R}_+}f(x)\pi(dx) \quad a.s.
 \]
as $t \to \infty$, where $\pi$ is the unique invariant probability measure of the BAJD.
\end{corollary}
\begin{proof}
The above convergence follows from Corollary \ref{coroinv} and \cite[Theorem 20.21]{MR1876169}.
\end{proof}
\section{Exponential ergodicity of the BAJD}
Let $\|\cdot\|_{TV}$ denote the total-variation norm for signed measures on $\mathbb{R}_+$, namely 
\[
\|\mu\|_{TV}=\sup_{A \in \mathcal{B}(\mathbb{R}_+)} \{|\mu(A)| \}.\]

The total variation norm for signed  measures on $\mathbb{R}_+$ is a special case of the  the norm $\|\cdot\|_h$, which is defined by 
\[\|\mu\|_h=\sup_{|g|\leq h}\Big|\int_{\mathbb{R}_+} g d\mu \Big|\]
 for a function $h $ on $\mathbb{R}_+$ with $h \ge 1$. Obviously it holds $ \|\mu\|_{TV} \le \|\mu\|_h $, given that $h \ge 1$.
 
Let $P^{t}(x,\cdot):=P_x(X_t\in \cdot)$ be the distribution of the BAJD process $X$ at time $t$ given that $X_0=x$ with $x \ge 0$. In this section we will find a function $h\ge1$ such that the BAJD is $h$-exponentially ergodic, namely there exist a constant $\beta \in (0,1)$ and a finite-valued function $B(\cdot)$ such that
\[
\|P^{t}(x,\cdot) -\pi\|_{h}\le B(x)\beta^t, \quad \forall t>0,  \quad x \in \mathbb{R}_+,
\]
where $\pi$ is the unique invariant probability measure of the BAJD.

\smallskip  We first show the existence of a Foster-Lyapunov function, which is a sufficient condition for the exponential ergodicity to hold. Let $\mathcal{A}$ denote the generator of the BAJD, then
\[
\mathcal{A}g(x)=\frac{1}{2}\sigma^2xg''(x)+(a\theta-ax)g'(x)+cd\int_{\mathbb{R}_{+}}\big(g(x+y)-g(x)\big)e^{-dy}dy
\]
for $g \in D(\mathcal{A})$.
\begin{lemma}\label{Lyapunov}
The function $V(x)=\exp(\gamma x)$ with small enough $\gamma>0$ is a Foster-Lyapunov function for the BAJD, namely there exist constants $k,M \in (0,\infty)$ such that 
\[
\mathcal{A}V(x) \le -k V(x)+M, \quad \forall x \in \mathbb{R}_{+}.
\]
\end{lemma}
\begin{proof}We have 
\begin{align*}
\mathcal{A}V(x)=&\frac{1}{2}\sigma^2\gamma^2xe^{\gamma x}+(a\theta-ax)\gamma e^{\gamma x}+cd\int_{\mathbb{R}_{+}}\big(e^{\gamma (x+y)}-e^{\gamma x}\big)e^{-dy}dy \\
=& \frac{1}{2}\sigma^2\gamma^2xe^{\gamma x}+(a\theta-ax)\gamma e^{\gamma x}+\frac{c\gamma}{d-\gamma} e^{\gamma x} \\
=& \gamma e^{\gamma x} \cdot \Big( \big(\frac{1}{2}\sigma^2\gamma-a\big)x+a\theta+\frac{c}{d-\gamma} \Big).
\end{align*}
If $\gamma>0$ is small enough, then $\sigma^2\gamma/2-a<0$ and there exists $x_0>0$ with 
\[
\big(\frac{1}{2}\sigma^2\gamma-a\big)x_0+a\theta+\frac{c}{d-\gamma}=-\frac{1}{\gamma}.
\]
Thus  we have for $x \in [x_0, \infty)$
\[
\mathcal{A}V(x) = \gamma e^{\gamma x} \cdot \Big( \big(\frac{1}{2}\sigma^2\gamma-a\big)x+a\theta+\frac{c}{d-\gamma} \Big) \le -e^{\gamma x}
\] 
and for $x \in [0, x_0]$
\[
\mathcal{A}V(x) = \gamma e^{\gamma x} \cdot \Big( \big(\frac{1}{2}\sigma^2\gamma-a\big)x+a\theta+\frac{c}{d-\gamma} \Big) \le \gamma e^{\gamma x_0} \cdot (a\theta+\frac{c}{d-\gamma} ) .
\]
It follows for all $x \in \mathbb{R}_{+}$ 
\[
\mathcal{A}V(x) \le -e^{\gamma x}+e^{\gamma x_0}+\gamma e^{\gamma x_0} \cdot (a\theta+\frac{c}{d-\gamma} ) \le -V(x)+M
\]
with 
\[M:=e^{\gamma x_0}+\gamma e^{\gamma x_0} \cdot \big(a\theta+\frac{c}{d-\gamma} \big) <\infty.
\]
\end{proof}

\begin{lemma} \label{Lyapunov2}  Let the constants $\gamma$, $k$ and $M$, as well as $V(x)=\exp(\gamma x)$ be the same as in Lemma \ref{Lyapunov}. Then the BAJD satisfies 
\[
E_x[V(X_t)] \le e^{-kt}V(x)+\frac{M}{k}
\]
or equivalently
\[
\int_{\mathbb{R}_+}V(y)p(t,x,y)dy \le e^{-kt}V(x)+\frac{M}{k}
\]
for all $x \in \mathbb{R}_+, \ t>0$.
\end{lemma}
\begin{proof}
Let $V(x)=\exp(\gamma x)$ and $g(x,t):=V(x) \cdot \exp(kt)=\exp(\gamma x+kt)$, where the constants $\gamma$ and $k$ are the same as in Lemma \ref{Lyapunov}. Then $g_x=\gamma  \exp(\gamma x+kt)$, $g_{xx}=\gamma^2 \exp(\gamma x+kt)$ and $g_t=k \exp(\gamma x+kt)$. By applying It\^o's formula and then taking the expectation, we get for all $x \in \mathbb{R}_+, \ t>0,$
\begin{align*}
&e^{kt}E_x[V(X_t)]-V(x) \\
=&E_x[g(X_t,t)]-E_x[g(X_0,0)]\\
=&E_x\Big[\int_0^t\big(e^{ks} \cdot\mathcal{A}V(X_s)+ke^{ks} \cdot V(X_s)\big)ds\Big]  \\
\le &E_x\Big[\int_0^t\big(e^{ks} \cdot \big(-kV(X_s)+M\big)+ke^{ks} \cdot V(X_s)\big)ds\Big] \\
= & \mathbb{E}_x\Big[\int_0^tMe^{ks}ds\Big] = \frac{M}{k}e^{kt}- \frac{M}{k}\le  \frac{M}{k}e^{kt}.
\end{align*}
Thus for $x \in \mathbb{R}_+, \ t >0,$
\[
E_x[V(X_t)] \le e^{-kt}V(x)+\frac{M}{k}.
\]
\end{proof}

Applying the main results of \cite{MR1234295} and  Lemma \ref{Lyapunov2}, we get the following theorem.

\begin{theorem} \label{theoremexp}
Let $h(x):=1+\exp(\gamma x)$ with the constant $\gamma >0$ small enough. Then the BAJD is $h$-exponentially ergodic,  namely there exist constants $\beta \in (0,1) $ and $C\in(0,\infty)$ such that 
\begin{equation}\label{exerjcir}
\|P^{t}(x,\cdot) -\pi\|_{h}\le C\big(e^{\gamma x}+1\big)\beta^t, \quad  t>0,  \quad x \in \mathbb{R}_{+}.
\end{equation}
\end{theorem}
\begin{proof}
Basically, we follow the proof of \cite[Theorem 6.1]{MR1234295}. For any $\delta>0$ we consider the $\delta$-skeleton chain $Y^{\delta}_n:=X_{n\delta},  \ n \in \mathbb{Z}_+$. Then $(Y^{\delta}_n)_{ n \in \mathbb{Z}_+}$ is a Markov chain with transition kernel $p(\delta,x,y)$ on the state space $\mathbb{R}_+$ and the law of the $Y_n$ (started from $Y_0=x$) is given by $P^{\delta n}(x, \cdot)$. It is easy to see that invariant measures for the BAJD process $(X_t)_{t\ge 0}$ are also invariant measures for $(Y^{\delta}_n)_{ n \in \mathbb{Z}_+}$. Thus the probability measure $\pi$ in Corollary \ref{coroinv} is also an invariant probability measure for the chain $(Y^{\delta}_n)_{ n \in \mathbb{Z}_+}$. 

Let $V(x)=\exp(\gamma x)$ be the same as in Lemma \ref{Lyapunov2}. It follows from  the Markov property and Lemma \ref{Lyapunov2} that 
\[
E_x[V(Y_{n+1})|Y_0, Y_1, \cdots, Y_n] = \int_{\mathbb{R}_+}V(y)p(\delta,Y_n,y)dy \le e^{-\delta k}V(Y_n)+\frac{M}{k},
\]
where $k$ and $M$ are positive constants. If we set $V_0:=V$ and $V_n:=V(Y_n)$, $n\in \mathbb{N}$, then 
\[
E_x[V_1] \le  e^{-\delta k}V_0(x)+\frac{M}{k}
\]
and 
\[
E_x[V_{n+1}|Y_0, Y_1, \cdots, Y_n]  \le e^{-\delta k}V_n+\frac{M}{k}, \quad n\in \mathbb{N}.
\]

Now we proceed to show that the chain $(Y^{\delta}_n)_{ n \in \mathbb{Z}_+}$ is $\lambda$-irreducible, strong aperiodic, and all compact subsets of $\mathbb{R}_+$ are petite for the chain $(Y^{\delta}_n)_{ n \in \mathbb{Z}_+}$. 

``$\lambda$-irreducibility": We show that the Lebesgue measure $\lambda$ on $\mathbb{R}_+$ is an irreducibility measure for $(Y^{\delta}_n)_{ n \in \mathbb{Z}_+}$. Let $A\in \mathcal{B}(\mathbb{R}_+)$ and $\lambda(A)>0$, then 
\[
P[Y_1 \in A|Y_0 =x]=P_x[X_{\delta}\in A]=\int_{A}p(\delta,x,y)dy>0,
\]
since $p(\delta,x,y)>0$ for any $x \in \mathbb{R}_+$ and $y>0$. This shows that the chain $(Y^{\delta}_n)_{ n \in \mathbb{Z}_+}$ is irreducible with $\lambda$ being an irreducibility measure.

``Strong aperiodicity" (see \cite[p.561]{MR1174380} for a definition):  To show the strong aperiodicity of $(Y_n)_{n \in \mathbb{N}_0}$, we need to find a set $C \in \mathcal{B}(\mathbb{R}_+)$, 
a probability measure $\nu$ with $\nu(C)=1$, and $\epsilon >0$ such that 
\begin{equation}\label{ergo1}
L(x,C)>0, \qquad x \in \mathbb{R}_+
\end{equation}
and 
\begin{equation}\label{ergo2}
P_x(Y_1\in A) \ge \epsilon \cdot \nu(A), \quad x \in C, \quad A \in \mathcal{B}(\mathbb{R}_+),
\end{equation}
where $L(x,C):=\mathbb{P}_x(Y_m \in C \ \rm{for} \ \rm{some} \ \it{m} \in \mathbb{N})$. To this end set $C:=[0,1]$ and $g(y):=\inf_{x \in [0,1]} p(\delta,x,y)$, $y>0$. Since for fixed $y>0$ the function $p(\delta,x,y)$ strictly positive and continuous in $x \in [0,1]$, thus we have $g(y)>0$ and $0<\int_{(0,1]}g(y)dy\le 1$. Define 
\[
\nu(A):=\frac{1}{\int_{(0,1]}g(y)dy}\int_{A \cap(0,1]}g(y)dy, \qquad A \in \mathcal{B}(\mathbb{R}_+).
\]
Then for any $x\in[0,1]$ and $A \in \mathcal{B}(\mathbb{R}_+)$ we get 
\[
P_x(Y_1\in A)=\int_{A}p(\delta,x,y)dy \ge \int_{A \cap(0,1]}g(y)dy =\nu(A) \int_{(0,1]}g(y)dy,
\]
so (\ref{ergo2}) holds with $\epsilon:=\int_{(0,1]}g(y)dy$. 

Obviously 
\[
L(x,[0,1]) \ge P_x(Y_1\in [0,1])= P_x(X_{\delta} \in [0,1])=\int_{[0,1]}p(\delta,x,y)dy>0
\]
for all $x \in \mathbb{R}_+$, which verifies (\ref{ergo1}).

``Compact subsets are petite": We have shown that $\lambda$ is an irreducibility measure for $(Y^{\delta}_n)_{ n \in \mathbb{Z}_+}$. According to \cite[Theorem 3.4(ii)]{MR1174380}, to show that all compact sets are petit, it suffices to prove the Feller property of $(Y^{\delta}_n)_{ n \in \mathbb{Z}_+}$, but this follows from the fact that $(Y^{\delta}_n)_{ n \in \mathbb{Z}_+}$ is a skeleton chain of the BAJD process $(X_t)_{ t  \ge0}$, which possess the Feller property.

According to \cite[Theorem 6.3]{MR1174380} (see also the proof of \cite[Theorem 6.1]{MR1174380}), the probability measure $\pi$ is the only invariant probability measure of the chain $(Y^{\delta}_n)_{ n \in \mathbb{Z}_+}$ and there exist constants $\beta \in (0,1) $ and $C\in(0,\infty)$ such that 
\[
\|P^{\delta n}(x,\cdot) -\pi\|_{h}\le C\big(e^{\gamma x}+1\big)\beta^n, \quad   n \in \mathbb{Z}_+,  \quad x \in \mathbb{R}_{+},
\]
where $h:=V+1=\exp(\gamma x)+1$.

Then we can proceed as in \cite[p.536]{MR1234295} and get the inequality (\ref{exerjcir}). 
\end{proof}

Since $ \|\mu\|_{TV} \le \|\mu\|_h $, it follows immediatly the follwing corollary.
\begin{corollary} \label{Cexpergodic}
The BAJD is exponentially ergodic, namely  there exist constants $\beta \in(0,\infty) $ and $C\in (0,\infty)$ such that 
\begin{equation}
\|P^{t}(x,\cdot) -\pi\|_{TV}\le C\big(e^{\gamma x}+1\big)\beta^t, \quad \forall t >0,  \quad x \in \mathbb{R}_{+}.
\end{equation}
\end{corollary}

\begin{remark} In \cite{MR3167406} we proved that the CIR process is positive Harris recurrent. If we allow the parameter $c=0$, then all results of this section still hold and thus are also true for the CIR process.  In particular Theorem \ref{theoremexp} is also true for the CIR process. In this case the unique invariant probability measure of the CIR process is the Gamma distribution $\Gamma\big(2a\theta/\sigma^2, \sigma^2/(2a)\big)$ and has the characteristic function $\big(1-\sigma^2u/(2a)\big)^{-2a\theta/\sigma^2}$.
\end{remark}

\par\bigskip\noindent
{\bf Acknowledgements.} The research was supported by the Research Programm "DAAD - Transformation: Kurzma{\ss}nahmen 2012/13 Programm". This research was also carried out with the support of CAS - Centre for Advanced Study, at the Norwegian Academy of Science and Letter, research program SEFE.

\bibliographystyle{gSSR}

\end{document}